\documentclass[11pt]{amsart}
\usepackage{amsmath,amssymb,amsthm,fancyhdr}

\usepackage[dvips]{graphicx}
\usepackage{color}
\usepackage{wasysym}

\theoremstyle{plain}

\usepackage{amsfonts}
\usepackage{amssymb}
\usepackage{amsmath}
\usepackage{amsthm}
\usepackage{mathrsfs} 
\usepackage{latexsym}
\usepackage{color}
\usepackage{comment}
\usepackage{cancel}
\usepackage[colorlinks=true]{hyperref}
\hypersetup{urlcolor=blue, citecolor=red, linkcolor=blue}

\numberwithin{equation}{section}

\newcommand{\N}{{\mathbb N}}

\newcommand{\R}{{\mathbb R}}

\definecolor{blu}{rgb}{0,0,1}

\newcommand{\e}{\varepsilon}

\def\r{\mathbb R}

\DeclareMathOperator{\Hyperg}{\mbox{ }_2 F_{1}}
\DeclareMathOperator{\supp}{supp}

\newcommand*{\abs}[1]{\left\vert #1\right\vert}
\newcommand*{\norm}[1]{\left\Vert #1\right\Vert}
\newcommand{\Sph}{\mathbb{S}}

\newcommand{\hsp}{\hspace{0.2cm}}

\newtheorem{theorem}{Theorem}[section]
\newtheorem{definition}[theorem]{Definition}
\newtheorem{proposition}[theorem]{Proposition}
\newtheorem{remark}[theorem]{Remark}
\newtheorem{lemma}[theorem]{Lemma}

\theoremstyle{definition}

\begin{document}

	\title[Singular solutions to the fractional Yamabe problem]{Qualitative properties of singular solutions to the fractional Yamabe problem}
	
	\author{Sergio Cruz-Bl\'{a}zquez \textsuperscript{1}}
	\address{Sergio Cruz-Bl\'{a}zquez, 
		\newline Università degli Studi di Bari \\ Dipartimento di Matematica \\ Via Edoardo Orabona 4 70125 Bari BA, Italy.}
		\email{sergio.cruz@uniba.it}
	
	\author{Azahara DelaTorre \textsuperscript{*}}	
	\address{Azahara DelaTorre,
		\newline Sapienza Universit\`{a} di Roma\\
		Dipartimento di Matematica Guido Castelnuovo\\
	Facolt\`a Scienze matematiche, fisiche e naturali\\
	Piazzale Aldo Moro, 5, 00185 Roma RM, Italy.
}
		\email{azahara.delatorrepedraza@uniroma1.it }
	
	\author{David Ruiz}
	\address{David Ruiz
		\newline IMAG, Universidad de Granada\\
		Departamento de An\'alisis Matem\'atico\\
		Campus Fuentenueva\\
		18071 Granada, Spain.}
	\email{daruiz@ugr.es}

\thanks{(*) Corresponding Author. \newline
	\textsubscript{1} Present Address: Universidad de Granada, Departamento de An\'alisis Matem\'atico, Campus Fuentenueva 18071 Granada, Spain. Email: sergiocruz@ugr.es \newline
	S.C. is supported by the project \textit{Quantitative and qualitative aspects of nonlinear PDEs} from the University of Campania Luigi Vanvitelli, funded under the MIUR bando 2017 PRIN - JPCAP. Part of this work was carried out during his visit to the University ``Sapienza Universit\`{a} di Roma'', to which he is grateful. \newline A. DlT acknowledges financial support from the Spanish Ministry of Science and Innovation (MICINN) through the grant Juan de la Cierva-Incorporaci\'{o}n 2018 with ref. IJC2018-036320-I and from the Spanish Government through the grant MTM2017-85757-P (MICIU). She is also supported by  Fondi Ateneo -- Sapienza Università di Roma. \newline
The three autors have been supported by the FEDER-MINECO Grant PID2021-122122NB-I00 and by J. Andalucia (FQM-116). A. DlT and D. R. also acknowledge financial support from the Spanish Ministry of Science and Innovation (MICINN), through the \emph{IMAG-Maria de Maeztu} Excellence Grant CEX2020-001105-M/AEI/10.13039/501100011033.}

	\date{}
	\maketitle

	\begin{abstract} In this paper we are interested in the qualitative properties of the solutions to the fractional Yamabe problem in $\R^n$ which present an isolated singularity. In particular, we prove that the Morse index of any such solution is infinity. The proof uses an Emden Fowler type transformation, so that we can pass to a nonlocal 1D problem posed in $\R$. 
	
\end{abstract}

\medskip


\medskip 


\section{Introduction and statement of the results}

The problem of finding a metric conformal to a given one with positive constant scalar curvature receives the name of Yamabe problem, and is a classic problem in Geometric Analysis. In the case of $\R^n$ ($n \geq 3$), if the conformal metric is written as $g_u= u^{\frac{4}{n-2}}|dx|^2$, we are led to the search of positive solutions to the equation:

$$ - \Delta u = u^{\frac{n+2}{n-2}} \mbox{ in } \R^n.$$
It is well known that the only positive smooth solutions to such problem are the \emph{Talentian} functions, which correspond to the metrics coming from the stereographic projection, composed with dilations and translations of the euclidean space.

Of particular interest is the existence of such conformal metrics with singularities: for instance, one can consider the problem:

\begin{equation} \label{presing} 
	 - \Delta u = u^{\frac{n+2}{n-2}} \mbox{ in } \R^n \setminus\{0\}, \ \lim_{x \to 0} u(x)= +\infty.
  \end{equation}
It was shown by Caffarellli, Gidas and Spruck \cite{CafGidSpr} that positive solutions of \eqref{presing} must be radially symmetric and the solution can be written as 
\[
u(x)=|x|^{\frac{-(n-2)}{2}}v(|x|),
\]
where $v$ is a function bounded between two positive constants. Writing the equation for the new function $v(r)$ and using the Emden Fowler change of variable ($|x|=r=e^{-t}$) the problem reduces to a second order ODE, that can be explicitly solved \cite{Sch2,Sch3}. We remark that the analysis of its phase portrait shows that all solutions must be periodic in the $t$ variable.

In the last years, nonlocal problems have captured a lot of attention and different generalizations of the known curvatures to the nonlocal setting have been developed. In particular, using the relation between scattering operators of asymptotically hyperbolic metrics and  Dirichlet-to-Neumann operators of degenerate elliptic problems (see \cite{GraZwo}), a definition of a one-parameter family of nonlocal operators was introduced by Gonz\'alez and Chang in \cite{ChaGon}. Correspondingly, a one-parameter family of intrinsic curvatures $Q_s$ with good conformal properties are defined. 

The fractional or nonlocal Yamabe problem has been posed in parallel to the classic one: finding a complete conformal metric with constant fractional curvature. Solvability of the problem was first considered by Gonz\'alez and Qing \cite{GonQin} and, assuming some dimensional and geometric properties on the manifold, it was proven in \cite{GonQin,GonWan,KimMusWei}. 
The most general case was proven by Kim, Musso and Wei \cite{KimMusWei}, assuming the validity of the fractional positive mass conjecture. 
 Let us observe that the fractional Yamabe problem with $s=\frac{1}{2}$ is deeply related to the so-called Escobar problem, which is another analogue of the Yamabe problem for manifolds with boundaries. Being more specific, the solutions to the problem:
\[
 	\left\{\begin{split}
 		&{\Delta u =0}
 		\mbox{ in }\mathbb R^n_+ \\
 		& {\frac{\partial u}{\partial \nu}=-Cu^{\frac{n}{n-2}} }	\mbox{ on }\mathbb R^{n-1},
 	\end{split}
 	\right.
\]
give rise to flat metrics in $\R^n_+$ with constant mean curvature of the boundary. And such problem is equivalent (via the well known extension for the fractional Laplacian) to the nonlocal equation:
 
\[
 	(-\Delta)^{s} u=C u^{\frac{N+2s}{N-2s}} 	\mbox{ on }\mathbb R^N,
\]
 with $N=n-1$, $s=\tfrac{1}{2}$.

When we allow the presence of singularities, the existence of complete Yamabe metrics depends on the geometry of the ambient manifold. In particular, there are upper bounds on the Hausdorff dimension of the singular set, which are explicit for the sphere (see \cite{SchYau}). In the nonlocal setting, a necessary condition was found in \cite{GonMazSir} (see also \cite{ACDFGW, ChaDlT2}).

In this paper we consider solutions to the fractional Yamabe problem in $\R^n$ with one isolated singularity, i.e.,

\begin{equation}\label{eq:R}(-\Delta)^{s}u=u^{\frac{n+2s}{n-2s}} \quad \mbox{ in }  \r^n\setminus\{0\},\quad \lim_{x\to 0} u(x) = +\infty.\end{equation}
Here $s \in (0,1)$, $n >2s$ and $(-\Delta)^s$ stands for the fractional laplacian, i.e., 
the integro-differential operator defined by \[(-\Delta)^{s} u(x)=c_{n,s} \, \text{P.V.} {\int_{\R^n}}\frac{u(x)-u(y)}{|x-y|^{n+2s}}\,dy,
\qquad
c_{n,s}
=\dfrac{
	2^{2s}
	\Gamma(\frac{n+2s}{2})
}{
	\Gamma(2-s)
	\pi^{\frac{n}{2}}
}s(1-s).
\]
Solutions to \eqref{eq:R} give rise to metrics $g_u:=u^{\frac{4}{n-2s}}|dx|^2$
which have constant fractional curvature. According to \cite{CJSX} (see also \cite{CheLiOu}), all singular solutions for this equation are radially symmetric and satisfy:

\begin{equation} \label{bound} c_1|x|^{-\frac{n-2s}{2}}\leq u(x)\leq c_2|x|^{-\frac{n-2s}{2}}, \ 0 < c_1 <c_2. \end{equation} Moreover we have the following explicit solution, see \cite[Proposition 2.7]{DG}:
\begin{equation} \label{singsol} u_0(x)=\kappa^{-1}_{n,s}|x|^{\frac{-(n-2s)}{2}}, \text{ where } \kappa_{n,s}=2^{2s}\left(\frac{\Gamma(\frac{n+2s}{4})}{\Gamma(\frac{n-2s}{4})}\right)^2. \end{equation}

The main scope of this paper is to obtain qualitative properties of the set of solutions of \eqref{eq:R}. Our first result is the following.
\begin{theorem}\label{intersect1}
	Let $u\in C^2_{loc}(\r^n \setminus \{0\})$ be any solution of \eqref{eq:R}. If $u(x) \geq u_0(x)$ or $u(x) \leq u_0(x)$ for all $x \in \r^n \setminus \{0\}$, then $u(x) =u_0(x)$.
\end{theorem}
We are also concerned with the Morse index computation of the singular solutions. In this regard, we prove the following theorem:

\begin{theorem}\label{th_morse_index}
	Let $u\in C^2_{loc}(\r^n \setminus \{0\})$ be any singular solution of \eqref{eq:R}. Then its Morse index is infinity. 
\end{theorem}

The main motivation of Theorem \ref{th_morse_index} is the blow-up analysis of solutions to the fractional Yamabe problem. Indeed, in many circumstances, one tries to prove compactness of solutions via a contradiction argument. If the sequence of solutions is unbounded, a rescaling argument yields a solution posed in the euclidean space. Under certain hypotheses of the rescaling, the limit solution could be singular at the origin. If one has a bound on the Morse index of the blowing-up solutions, then the limit solution should have finite Morse index, and Theorem \ref{th_morse_index} would give a contradiction. 
	
%

The proofs of Theorems \ref{intersect1} and \ref{th_morse_index} follow at first the strategy of the local case; indeed, the estimate \eqref{bound} motivates the change of unknown $v(r)=\kappa_{n,s}r^{\frac{n-2s}{2}}u(r)$, with $r=|x|$. In polar coordinates $(r,\theta)\in (0,+\infty)\times \Sph^{n-1}$, the fractional Laplacian satisfies the following conformal property (see \cite{ChaGon, DG}):
\[\label{conformal_invariance}
	(-\Delta^s)\left(r^{-\frac{n-2s}{2}}v(r)\right)=r^{-\frac{n+2s}{2}}\left(P v(r)+v(r)\right),
\]
where
\[
	Pv(t):= \text{P.V.} \, \int_{\r}(v(t)-v(\tau))  K(t-\tau)\,d\tau,\]with 
\begin{equation}\label{eq:kernel} {K}(t)=	\gamma_{n,s} \,e^{-\frac{n+2s}{2}|t|}\Hyperg\big( \tfrac{n+2s}{2},1+s;\tfrac{n}{2};e^{-2|t|}\big), \ \gamma_{n, s} >0.
\end{equation}
Here $\Hyperg$ represents the hypergeometric function. In particular, we have the following asymptotic behavior for the kernel (see for instance \cite{survey, DDGW}):
\begin{equation}\label{beh_kernel}
	K(\xi)\asymp
	\begin{cases}
		|\xi|^{-1-2s} &\mbox{ as }|\xi|\to 0,\\
		e^{-\frac{n+2s}{2}|\xi|} &\mbox{ as }|\xi| \to +\infty.\\
	\end{cases}
\end{equation}
Via the Emden-Fowler change of variable $r= e^{-t}$ we are led now with the fractional problem:

\begin{equation}\label{eq:cylinder}
	Pv+v=v^{\frac{n+2s}{n-2s}},\quad  \text{ in } \r.
\end{equation}
 Moreover, by \eqref{bound}, the solution $v$ satisfy:

\begin{equation}\label{boundedness} 0< c_1 \leq v(t) \leq c_2, \quad \forall \ t \in \R. \end{equation}
Observe that the solution $v=1$ of \eqref{eq:cylinder} corresponds to the solution \eqref{singsol}. It is worth to point out, however, that the classical methods for ODE's cannot be applied here, and a “phase portrait” seems not possible (see \cite{DG, survey}).

Little is known about the qualitative properties of the solutions of \eqref{eq:cylinder}. Inspired by the classical case, one could conjecture that all such solutions $v$ are periodic in $t$: this is by now a major open problem. Periodic solutions have been constructed in \cite{DDGW} (see also \cite{ACDFGW,survey}). 

In order to prove Theorem \ref{intersect1} it suffices to show that any solution $v$ must intersect the constant solution 1. This is done by using the first eigenfunction of the operator $P$ in bounded large domains, and using a convenient comparison principle. 

The proof of Theorem \ref{th_morse_index} reduces to show that the Morse index of any solution $v$ of \eqref{eq:cylinder} is also infinity. In order to show this, we discuss two different cases. 

We say that a solution $v$ satisfies the Oscillation Condition if there exist constants $M>0$ and $\epsilon>0$ such that for any interval $I=[a,b]$ with length $|b-a|=M$, then:
\[ \max \{ v(t):\ t \in I\} > 1+\e \ \text{ and } \ \min\{ v(t): \ t \in I\} < 1- \e.\]
 
If the solution $v$ does not satisfy the Oscillation Condition, we prove that there exists a sequence $\tau_n \in \R$ such that $v(\cdot - \tau_n) \to 1$ in $C^k_{loc}$. The proof of  this fact uses Theorem \ref{intersect1}. Moreover, one can check that the solution $1$ has infinite Morse index, and in this way we conclude.

If, instead, the solution $v$ satisfies the Oscillation Condition, we prove that one can use the function $|v'|$, conveniently truncated, to make the quadratic form become negative. An iterative argument gives the existence of many such truncations, and this implies that the Morse index is infinity. 

The rest of the paper is organized as follows. In Section 2 we are concerned with Theorem \ref{intersect1}. For that, some preliminaries are in order. In particular, we need to consider the associated eigenvalue problem in bounded domains. In Section 3 we first give a definition of the Morse index associated to problem \eqref{eq:R}, and we show its relation with the Morse index of solutions to problem \eqref{eq:cylinder}. We then use this relation to prove that the latter is infinity.

 \section{Proof of Theorem \ref{intersect1}}
 
In this section we prove Theorem \ref{intersect1} by showing that any solution $v$ for \eqref{eq:cylinder} has to intersect the constant solution $v_0\equiv 1.$ To start with, let us recall the definition of the fractional Sobolev space $H^s(\Omega)$:

\[ H^s(\Omega)= \left\{u \in L^2(\Omega): \ \frac{|u(x)-u(y)|}{|x-y|^{s+n/2}} \in L^2(\Omega \times \Omega) \right\}, \ s \in (0,1),\]
equipped with the norm
\[\norm{u}_{H^s(\Omega)}=\left({\norm{u}_{L^2(\Omega)}^2}+\int_\Omega\int_\Omega\frac{\abs{u(x)-u(y)}^2}{\abs{x-y}^{n+2s}}\,dx\,dy\right)^\frac{1}{2}.\]
As has been anticipated in the introduction, most of the time we will consider Sobolev Spaces in dimension 1.

\begin{lemma} Let $I \subset \R$ be any interval, $u \in  L^{\infty}(\R)$ such that $u|_I \in C^{2}(I)$. Then $P u \in C(J)$ for any $J \subset \subset I$, where
	\[
		Pu(t):= \text{P. V. } \, \int_\R \left(u(t)-u(\tau)\right) K(t-\tau)d\tau, \ t \in J.
	\]
Here $K$ is given in \eqref{eq:kernel} (and hence satisfies \eqref{beh_kernel}).
		 
\end{lemma}

The proof of this lemma is standard, using the asymptotic properties of the kernel $K$ (see \eqref{beh_kernel}), and is left to the reader.

The operator $P$ satisfies a maximum principle, which is stated and proven below for the sake of completeness.

\begin{proposition}[Maximum Principle] \label{mp} Let $I \subset \R$ be an open interval, $u \in  L^{\infty}(\R)$ such that $u|_I \in C^{2}(I)$. Assume that $u(t) \geq 0$ for all $t \in \R$, and $P u(t) \geq 0$ for any $t \in I$. Then, either $u(t) >0$ for all $t \in I$ or $u(t) = 0$ a.e. $t\in \R$. \end{proposition}

\begin{proof}
Assume $u \geq 0$, $u$ not identically equal to $0$, and assume that for some $t_0 \in I$, $u(t_0) =0$. Then,

$$P u(t_0)= P. V. \, \int_{\R} K(t_0-\tau) (u(t_0) - u(\tau)) \, d \tau= P. V. \, \int_{\R} K(t_0-\tau) (-u(\tau)) \, d \tau <0. $$
This proves the result.
\end{proof}

As anticipated in the introduction, the first eigenvalue of $P$ on bounded intervals will play a major role in the proof of Theorem \ref{intersect1}. On that purpose, we introduce the quadratic form
\begin{equation}\label{eq:T}
\mathcal{T}(v)=\frac{1}{2}\int_{\R} \int_{\R} K(t-\tau)(v(t)-v(\tau))^2\, dt \, d\tau,
\end{equation}
defined for every compactly supported $v\in H^{s}(\R)$. 
\begin{definition}\label{def-eigenvalue} Given any $M>0$, we define \textit{the first eigenvalue of $P$ in $[-M,M]$} as
\[
\lambda_1(M)=\inf \left\lbrace\frac{\mathcal{T}(v)}{\int_\R v(t)^2\,dt}: \ v \in H^s(\r) \setminus \{0\},\: \supp v\subset[-M,M]\right\rbrace.
\]
\end{definition}
The properties of $\lambda_1(M)$ that will be needed later are collected in the following lemma:
\begin{lemma}\label{eigenvalues}The following assertions hold true:
	\begin{enumerate}
		\item[(i)] $\lambda_1(M)>0$ for every $M>0$.
		\item[(ii)] There exists a continuous, nonnegative function $\phi_1\in H^{s}(\R)$ such that $\supp \phi_1\subset [-M,M]$, $\phi_1(t)>0$ for every $t\in(-M,M)$ and
		\[
		\frac{\mathcal{T}(\phi_1)}{\int_\R \phi_1(t)^2dt} = \lambda_1(M).
		\]
		Moreover, $\phi_1 \in C^{\infty}(-M,M)$ and it satisfies 
		\begin{equation}\label{phi1}
		P \phi_1=\lambda_1 \phi_1 \hsp\text{in}\hsp (-M,M).
		\end{equation}
		\item[(iii)] $	\lambda_1(M)\to 0$, as   $M\to +\infty$. 
		\end{enumerate}
	\end{lemma}
\begin{proof}

Observe that, by definition, $\lambda_1(M) \geq 0$. We first prove that $\lambda_1(M)$ is achieved by a nonnegative function $\phi_1$ for every $ M>0$. Take $v_k \in H^s(\r)$ with $\supp v_k \in [-M,M]$ such that $\int_\R v_k(t)^2\, dt =1$ and $\mathcal{T}(v_k) \to \lambda_1(M)$. By \eqref{beh_kernel}, there exists a constant $C_1=C_1(M)$ such that 
\begin{equation}\label{kernel-bound}
K(\xi) \geq C_1 |\xi|^{-1-2s} \hsp\text{for}\hsp |\xi| \leq M. 
\end{equation}
Using \eqref{kernel-bound} and the expression for $\mathcal{T}$ given in \eqref{eq:T},
\[
\mathcal{T}(v_k) + \norm{v_k}_{L^2(\R)}  \geq \min \{ C_1, 1 \}\| v_k \|^2_{H^s(\R)}.
\]
Therefore, $v_k$ is bounded in $H^s(\r)$ and up to taking a subsequence we can assume that $v_k \rightharpoonup \phi_1$ weakly in $H^s(\R)$ and strongly in $L^2(\R)$ (see \cite[Theorem 7.1]{hitchhiker}). A standard minimization argument permits us to conclude that $\phi_1$ is a minimizer. With this in mind, if $\lambda_1(M)=0$, then $\phi_1$ would be constantly equal to $0$, contradicting $\norm{\phi_1}_{L^2(\R)}=1$.

\medskip 

Now, testing the quotient with $\abs{\phi_1},$ we see that
\[\mathcal{T}(\abs{\phi_1})=\frac12\int_\R K(t-\tau )(|\phi_1(t)| - |\phi_1(\tau)|)^2d\tau dt \leq  \frac12\int_\R K(t-\tau)(\phi_1(t) - \phi_1(\tau))^2d\tau dt = \mathcal{T}(\phi_1).\]
Thus, we can assume that $\phi_1\geq 0$. \eqref{phi1} is the Euler-Lagrange equation for the functional $\mathcal{T}$ under the constraint $\int_\R v^2 = 1$, and the maximum principle (Proposition \ref{mp}) gives us $\phi_1(t)>0$ for all $t\in(-M,M)$. The regularity of $\phi_1$ is known (see \cite[Section $6.1$]{surveyRO} and the references therein), being $\phi_1\in C^0(\R)\cap C^1(-M,M)$. A bootstrap argument yields interior $C^{\infty}$ regularity.
\bigskip

Finally, we show that $\lambda_1(M)\to 0$, as $M\to +\infty.$ By \eqref{beh_kernel}, the following bound holds:
\[ K(\xi) \leq C |\xi|^{-1-2s} \hsp \text{for every } \xi \in \R.\] Then,
\begin{equation} \label{casi} \lambda_1(M) \leq C \mu_1(M)/2, \end{equation}
where:
\[
	\mu_1(M):=\inf \left\{\displaystyle \frac{\int_\R \int_\R (v(t) - v(\tau))^2 |t-\tau|^{-1-2s}\, dt \, d\tau}{\int_\R v(t)^2\,dt}: \ v \in H^s(\R) \setminus \{0\}, \   \text{ supp }v\subset[-M,M] \right\}.
\] 
It is clear that $\mu_1(M)$ is decreasing in $M$. Moreover, by the scaling properties of the terms in the definition of $\mu_1(M)$, we have that:

$$ \mu_1(M)= \mu_1(1) M^{-2s}.$$
From this inequality and \eqref{casi}, the claimed result follows.
	\end{proof}

We are now ready to prove Theorem \ref{intersect1}, which is an immediate consequence of the next result:

\begin{proposition} \label{prop1} Let $v$ a solution of \eqref{eq:cylinder} satisfying \eqref{boundedness}. Assume that $v(t) \geq 1$ (or $v(t) \leq 1$) for all $t \in \r$. Then $v(t)\equiv1$.
\end{proposition}

\begin{proof}
Assume first that $v\geq 1$. Observe that 
\begin{equation}\label{claimgeq1}
	 P (v-1) \geq \frac{4s}{n-2s}(v-1) \geq 0 \hsp\text{in } \R.
\end{equation}
This is a consequence of the concavity of the function $t\to t^{\frac{n+2s}{n-2s}}-t$ together with the fact that $ P 1 =0$.  By Proposition \ref{mp} we have that either $v=1$ or $v-1>0$. Let us see that the latter possibility yields a contradiction. 

By Lemma \ref{eigenvalues}, (3), it is possible to find $M>0$ such that $\lambda_1(M)<\frac{4s}{n-2s}$. We will denote by $\phi_1$ the corresponding eigenfunction, satisfying 
\begin{equation}\label{rel01}
 P\phi_1=\lambda_1(M)\phi_1 \hsp \text{in } (-M,M).
\end{equation}
Let us recall that $\phi_1 \in H^s(\R) \cap C^0(\R)$, $supp \ \phi_1 = [-M,M]$ and $\phi_1(t)>0$ for all $t \in (-M, M)$.
Then, we can define $\alpha >0$ as:

$$ 0< \sup \left \{ \frac{\phi_1(x)}{v(x) - 1}, \ x \in \R \right \} = \max \left \{ \frac{\phi_1(x)}{v(x) - 1}, \ x \in \R \right \}:= 1/\alpha.$$

By continuity, we have that $w :=  v- 1 - \alpha \phi_1 \geq 0$ but there exists $t_0 \in (-M,M)$ such that $w(t_0)=0$. We can now combine \eqref{rel01} with \eqref{claimgeq1} to conclude that, for any $t \in (-M,M)$, 

$$ P w (t)  \geq \frac{4s}{n-2s}(v-1) - \lambda_1(M) \alpha \phi_1 \geq \lambda_1(M) w(t) \geq 0.$$
But this is a contradiction with Proposition \ref{mp}, and we conclude.

Now, let us consider the case $v\leq 1$. Again by concavity, and taking into account \eqref{boundedness}, we observe that there exists $c>0$ such that \[v^{\frac{n+2s}{n-2s}}-v\leq c(v-1).\] Hence, \[ P(1-v)\geq c(1-v),\] and we can argue analogously as above by taking $M$ such that $\lambda_1(M)<c.$
\end{proof}

\section{Proof of Theorem \ref{th_morse_index}}

This section addresses the proof of Theorem \ref{th_morse_index}. First of all, we will specify the definition of the Morse Index in our setting. We will see that some of its fundamental properties carry over naturally to the nonlocal case, others, such as its behavior when passing to the limit, require more attention. 

 \begin{definition}\label{def:morse} Given $u$ a solution of \eqref{eq:R}, we define the associated quadratic form as: 	\begin{equation}\label{quadrn}
		\mathscr{Q}_u[\psi] = \frac{1}{c_{n,s}}\int_{\R^{n}\backslash\{0\}}\psi(-\Delta)^s\psi - \frac{n+2s}{n-2s}\int_{\R^n\backslash\{0\}}u^\frac{4s}{n-2s}\psi^2,
	\end{equation}
	which is well defined, at least, for functions $\psi \in C_0^{\infty}(\r^n \setminus \{0\})$. We define the Morse index of the solution $u$ as follows:
	\[\text{Ind}(u) = \sup \left\lbrace \dim(E): \ E \subset C_0^{\infty}(\r^n \setminus \{0\}) \mbox{ a vector space such that } \mathscr{Q}_u|_E \mbox{ is negative definite} \right\rbrace.\]
\end{definition}

Our intention is to show that the Morse index $Ind(u)$ of any solution $u$ of \eqref{eq:R} is infinity. To show this, it suffices to consider subspaces $E$ of radially symmetric functions of arbitrary dimension where $\mathscr{Q}_u|_E$ is negative definite. In this particular setting, we can work with the following alternative quadratic form:

\begin{lemma}\label{lemmaquad} Let $u(x)$ be a (radially symmetric, thanks to  \cite{CJSX}) solution of \eqref{eq:R}, and $\psi\in C^\infty_0(\R^n\backslash\{0\})$ a radially symmetric function. Then,
	\[
			\mathscr{Q}_u[\psi ]= \abs{\Sph^{n-1}}Q_v[\phi],
	\]
	where $v(t)= r^{\frac{n-2s}{2} } u(r)$, $\phi(t)=  r^{\frac{n-2s}{2} } \psi(r)$, $r=e^{-t}$ and
	\begin{equation}\label{radialquad}
		Q_v[\phi]= \frac 1 2 \int_\R \int_\R \left(\phi(t)-\phi(\tau)\right)^2K(t-\tau)\, d\tau dt+\int_\R\left(1-\frac{n+2s}{n-2s}v(t)^\frac{4s}{n-2s}\right)\phi(t)^2dt.
	\end{equation}

\end{lemma}

\begin{proof}
	
	We rewrite \eqref{quadrn} using polar coordinates and the change of variable $t=-\log r$ 
	\begin{equation}\label{qf-1}
		\begin{split}
		\mathscr{Q}_u[\psi]&= \frac{\abs{\Sph^{n-1}}}{c_{n,s}}\int_0^{\infty}\psi(r)(-\Delta)^s\psi(r)r^{n-1} \,dr- \frac{n+2s}{n-2s}\abs{\Sph^{n-1}}\int_{0}^{\infty}u(r)^\frac{4s}{n-2s}\psi^2(r)r^{n-1}\, dr\\&=\abs{\Sph^{n-1}}\int_\R e^{\frac{n-2s}{2}t}\phi(t)e^{\frac{n+2s}{2}t} \left(P\phi(t)+\phi(t)\right)e^{-(n-1)t}e^tdt\\&-\frac{n+2s}{n-2s} \abs{\Sph^{n-1}}\int_\R e^{2st}v(t)^\frac{4s}{n-2s}e^{(n-2s)t}\phi(r)^2e^{-(n-1)t}e^tdt \\[0.15cm]&= \abs{\Sph^{n-1}}\left(\int_\R\phi(t){P\phi(t)dt}+\int_\R \phi(t)^2dt-\frac{n+2s}{n-2s}\int_\R v(t)^\frac{4s}{n-2s}\phi(t)^2dt\right).
		\end{split}
	\end{equation}
We claim that:
\begin{equation}\label{qf-2}
\int_\R \phi(t) P\phi(t)dt = \frac{1}{2}\int_\R\int_\R  K(t-\tau)\left(\phi(t)-\phi(\tau)\right)^2d\tau dt.
\end{equation}
By definition,
\begin{align*}
\int_\R\phi(t)P\phi(t)dt &= \int_\R\phi(t)\,\left(\text{P.V.}\,\int_\R(\phi(t)-\phi(\tau))K(t-\tau)d\tau\right)dt \\ &=\int_\R\phi(t)\left(\lim_{\e\to0}\int_{\R\setminus(t-\e,t+\e)}(\phi(t)-\phi(\tau))K(t-\tau)d\tau\right)dt.
\end{align*}
Call \[F_\e(t):=\phi(t)\int_{\R\setminus(t-\e,t+\e)}(\phi(t)-\phi(\tau))K(t-\tau)d\tau.\]
Also, for $R>0$ large enough, define $A(t,\e,R)=[t-R,t+R]\setminus(t-\e,t+\e)$. Then, by the properties of the kernel $K$ and Taylor expansion,
\begin{align*}
\abs{F_\e(t)}&=\abs{\phi(t)}\abs{\phi'(t)\int_{A(t,\e,R)}K(t-\tau)(t-\tau)d\tau+\frac{1}2\int_{A(t,\e,R)}\phi''(\xi)K(t-\tau)(t-\tau)^2+C\,} \\ &\leq \abs{\phi(t)}\left(\norm{\phi''}_{L^\infty(\R)}\int_{t+\e}^{t+\e+R}\frac{\bar C}{\abs{t-\tau}^{2s-1}}d\tau+C\right) \leq C \abs{\phi(t)},
\end{align*}
for some $\xi$ depending on $t,\tau,\e$ and $R$. Since $\phi(t)$ is compactly supported, by the Dominated Convergence Theorem we can write
\[
\int_\R\phi(t)P\phi(t)dt = \lim_{\e\to0}\int_\R\int_{\R\setminus(t-\e,t+\e)}\phi(t)(\phi(t)-\phi(\tau))K(t-\tau)d\tau dt.
\]
Observe that the function $\phi(t)(\phi(t)-\phi(\tau))K(t-\tau)$ is integrable in $\{(t,\tau)\in\R^2:\:\abs{t-\tau}>\e\}$. Therefore, relabeling the integration variables and using Fubini's Theorem, one can see that
\begin{eqnarray*}
\lim_{\e\to0}\int_\R\int_{\R\setminus(t-\e,t+\e)}\hspace{-0.7cm}\phi(t)(\phi(t)-\phi(\tau))K(t-\tau)d\tau dt = -\lim_{\e\to0}\int_\R\int_{\R\setminus(t-\e,t+\e)}\hspace{-0.7cm}\phi(\tau)(\phi(t)-\phi(\tau))K(t-\tau)d\tau dt.
\end{eqnarray*}
Consequently, 
\begin{align*}
\int_\R\phi(t)P\phi(t)dt &= \lim_{\e\to0}\int_\R\int_{\R\setminus(t-\e,t+\e)}\hspace{-0.7cm}\phi(t)(\phi(t)-\phi(\tau))K(t-\tau)d\tau dt = \\&=  \lim_{\e\to0} \frac 1 2 \int_\R\int_{\R\setminus(t-\e,t+\e)}\hspace{-0.7cm}(\phi(t)-\phi(\tau))^2K(t-\tau)d\tau dt, 
\end{align*}	
and the claim follows because of integrability of the last term. 
\medskip 

Finally, we can combine \eqref{qf-1} and \eqref{qf-2} to conclude the proof of the Lemma.
\end{proof}
The above lemma motivates the following definition:

\begin{definition} Let us define:
	$$ind(v) = \sup \{ dim F: \ F \subset C_0^{\infty}(\r) \mbox{a vector space such that } Q_v|_F \mbox{ is negative definite} \}.$$
\end{definition}
Clearly, since we are restricting ourselves to radial functions, $Ind(u) \geq ind(v)$ where $v(t)= r^{\frac{n-2s}{2} } u(r)$, $r=e^{-t}$.

\begin{remark} \label{remark} A first remark is that the Morse index of the solutions $v$ as defined above is translation invariant. In other words, if $v$ is a solution of \eqref{eq:cylinder} and we define $v_a(t)= v(t - a)$, $a \in \R$, then $v_a$ is also a solution of \eqref{eq:cylinder} and $ind(v)=ind(v_a)$.
\end{remark}
The following result shows that for every convergent sequence of solutions of \eqref{eq:cylinder}, one can take a subsequence whose Morse Index stays above the Morse Index of its limit. There exist analogue results for the local case, but their proofs cannot be translated to the nonlocal framework immediately.

\begin{lemma}\label{lem:reg}
	Let $v_k$ be a sequence of solutions of \eqref{eq:cylinder} satisfying \eqref{boundedness} uniformly in $k$. Then, up to taking a subsequence, $v_k\to v_\infty$ in $C^h_{loc}(\R)$ for any $h \in \mathbb{N}$ and $v_\infty$ also solves \eqref{eq:cylinder}. Moreover,
	
	\begin{equation} \label{order} \liminf_{k \to + \infty} ind(v_k) \geq ind(v_{\infty}). \end{equation}
\end{lemma}

\begin{proof} 
	The $C^\infty-$regularity of the solutions of \eqref{eq:cylinder} is well-known, see \cite[\textsection 3]{DDGW} and \cite{hitchhiker}. In any given compact set $\Omega\subset \subset \R$, all the derivatives of $v_k$ are uniformly bounded and equicontinuous (see \cite[Remark 3.3]{DDGW}) and thus admit a uniformly convergent subsequence $v_k^{(l)}\to v_\infty^{(l)}$ in $\Omega$ by Ascoli-Arzel\`a Theorem (see also \cite[Th. 7.17]{rudinbook}). By a standard Cantor diagonalization argument, we have $v_k\to v_\infty$ in $C^\infty_{loc}(\R).$
	
	\medskip 
	
	Now, let us show that $v_\infty$ solves \eqref{eq:cylinder}. We will use the local $C^\infty$ convergence and the decay of the kernel given by \eqref{beh_kernel} to see that $ P v_k \to  P v_\infty$ a.e. in $\R$. Let us call $w_k=v_k-v_\infty$. Then, for every {fixed} $t\in \R:$
	
	\begin{equation} \label{regularity01}
	\begin{split}	
		{P} w_k(t) &= \text{P.V.}\int_\R K(t-\tau)((w_k(t)-w_k(\tau))d\tau \\ & \leq \text{P.V.}\int_I K(t-\tau)((w_k(t)-w_k(\tau))d\tau + C\int_{\R\backslash I} \left\vert w_k(t)-w_k(\tau)\right\vert e^{-\frac{n+2}{2}(\tau-t)}d\tau
	\end{split}
	\end{equation}
	where $I=[t-M,t+M]$ and $M>t+1$. By Taylor expansion, using the symmetry of the kernel, 
	\begin{equation}\label{regularity02}
	 \text{P.V.} \, 	\int_{I} (w_k(t)-w_k(\tau))K(t-\tau)d\tau \leq w_k'(t)\text{P.V.} \, \int_I (\tau-t)K(t-\tau)d\tau+C\norm{w_k''}_{L^\infty(I)}\int_I \frac{(\tau-t)^2}{\abs{t-\tau}^{1+2s}}.
	\end{equation}
	Notice that the first term in the right-hand side of \eqref{regularity02} is zero because of symmetry, and the second one is finite since $s< 1$. On the other hand,
	\begin{equation}\label{regularity03}
		\int_{\R\backslash I} \left\vert w_k(t)-w_k(\tau)\right\vert e^{-\frac{n+2}{2}(\tau-t)}d\tau \leq \abs{w_k(t)}\frac{4e^{\frac{-M(n+2)}{2}}}{n+2}+\int_{\R\backslash I}\abs{w_k(\tau)}e^{-\frac{n+2}{2}(\tau-t)}d\tau\to 0,
	\end{equation}
	by the Dominated Convergence Theorem, since $w_k\to0$ a.e. in $\R.$ Inserting \eqref{regularity02} and \eqref{regularity03} in \eqref{regularity01}, we conclude
	\[
		\abs{{P} w_k(t)} \leq C \left(\norm{w''_k}_{L^\infty(I)} +\abs{w_k(t)}\right) \to 0.
	\]
Let us now show \eqref{order}. Take $F \subset C_0^{\infty}(\r)$ a finite dimensional vector space such that $Q_{v_{\infty}}[\psi]<0$ for all $\psi \in F$. By compactness, we can show that $Q_{v_{\infty}}[\psi]<-\e<0$ for all $\psi \in F$, $\| \psi \|=1$, where $\| \cdot \|$ represents any norm in $F$.

Given $\psi \in F$, $\| \psi\|=1$, we have that from the definition of $Q_v$ given in \eqref{radialquad},
\[ Q_{v_k}[\psi]  \to Q_{v_{\infty}}[\psi].\]
Then, there exists $k_0 \in \N$ such that if $k \geq k_0$, $Q_{v_k}|_F$ is negative definite. As a consequence, $ind(v_k) \geq dim \ F$. This concludes the proof.
\end{proof}

The proof of Theorem \ref{th_morse_index} is divided into two cases, depending on the behavior of the solution $v$. This is determined by the following definition.
\begin{definition}\label{def:osc}
		Let $v$ be a solution for \eqref{eq:cylinder}. We say that $v$ satisfies the Oscillation Condition (OC) if

	\begin{equation}\tag{OC}\begin{split}
		\exists\ M>0, \ \epsilon>0: \ \forall a,b\in\R, \ b-a=M,\\ \quad \max \{ v(t):\ t \in [a,b]\} > 1+\e \ \text{ and } \ \min\{ v(t): \ t \in [a,b]\} < 1- \e.
		\end{split}\end{equation}

\end{definition}

In the next proposition we prove that if $v$ does not satisfy the above Oscillation Condition, its Morse index is infinity. 

\begin{proposition} \label{prop2} Let $v$ be a solution of \eqref{eq:cylinder} satisfying \eqref{boundedness}. If $v$ does not satisfy (OC), then $ind(v)=+\infty$.
\end{proposition}

\begin{proof}

	Take $M_k:=k$ and $\epsilon_k:=\frac{1}{k}$ for any $k \in \N$. Then there exists a sequence of points $t_k\in\r$ and intervals $I_k=[t_k-k/2,t_k+k/2]$ such that either
	\begin{equation} \label{maxomin}
		\max \{ v(t):\ t \in I_k\} \leq 1+\e_k, \quad  \text{or} \quad \min \{ v(t):\ t \in I_n\} \geq 1-\e_k.
	\end{equation}
	We suppose that the first possibility holds, the other case being analogous. Define a sequence of bounded functions by translation as  $v_k(t)=v(t-t_k)$; we can apply Lemma \ref{lem:reg} 
	to assert that there exists $v_0$ solution to \eqref{eq:cylinder} such that $v_k\to v_0$ in $C^h_{\text{loc}}(\r)$. Since $\max \{ v(t):\ t \in I_k\} \leq 1+\frac{1}{k}$  we conclude that $v_0(t) \leq 1$ for all $t \in \R$. By Proposition \ref{prop1}, we conclude that $v_0(t)=1$ for all $t \in \R$. We now use Lemma \ref{lem:reg} to conclude that 
	
	$$ \liminf_{k \to +\infty} ind(v_k) \geq  ind(1).$$
	By the invariance of the Morse index via translation (see Remark \ref{remark}), we have that actually $ind(v_k)=ind(v)$. 

\medskip 
	
The proof concludes by showing that $ind(1)=+\infty$. In this case the quadratic form reads:

\[Q_1[\phi]= \frac 1 2 \int_\R \int_\R \left(\phi(t)-\phi(\tau)\right)^2K(t-\tau)\, d\tau dt - \frac{4 \, s }{n-2s}\int_\R \phi(t)^2dt.\]
	
By Lemma \ref{eigenvalues}, we can take $M>0$ so that $\lambda_1(M) < \frac{4 \, s}{n-2s}$. If $\phi_1$ is the corresponding eigenfunction, we have that $ Q_1[\phi_1] <0.$ Using the density of $C^\infty_0(-M,M)$ in \newline $\left\lbrace u\in H^{s}(\R): \supp u\subset [-M,M]\right\rbrace$, we find $\phi \in C_0^{\infty}(-M,M)$ with 
\[Q_1[\phi]<0.\]

Take now $m \in \N$, and define $\psi_j(x)= \phi(x- j (d_m+2M))$, $j=1 \dots m$, where $d_m$ is a large positive constant that will be chosen later, and provides a lower bound for the minimal distance between the supports of $\psi_j$. Clearly, the functions $\psi_j$ form a linearly independent set in $C_0^{\infty}(\R)$. Given $0 \neq \psi = \sum_{j=1}^m \alpha_j \psi_j$, we have that:

$$ Q_1[\psi] = \sum_{j=1}^m \alpha_j^2 Q_1[\psi_j] + \sum_{i\neq j}^m \alpha_i \alpha_j A_1[\psi_i, \psi_j].$$
Here $A_1$ denotes the bilinear form:

\[
	A_1[\phi, \varphi]=	\frac 1 2 \int_\R \int_\R \left(\phi(t)-\phi(\tau) \right) \, \left(\varphi(t)-\varphi(\tau) \right)K(t-\tau)\, d\tau dt - \frac{4 \, s }{n-2s}\int_\R \phi(t) \varphi(t)dt.
\] 
Observe that for $i\neq j$, since $\psi_i$ and $\psi_j$ have disjoint supports:
\begin{align*}
 |A_1[\psi_i, \psi_j]| &\leq \int_{\text{supp}(\psi_i)}\abs{\psi_i(t)}\int_{\text{supp}(\psi_j)}\abs{\psi_j(\tau)}K(t-\tau)\,d\tau dt \\ &\leq K(\abs{i-j}(d_m+2M)) \left( \int_\R |\psi_i(t)| \, dt  \right) \left( \int_\R |\psi_j(\tau)| \, d \tau \right)  \to 0 \mbox{ if } d_m \to +\infty.
\end{align*}
Then, for sufficiently large $d_m$, we conclude that $Q_1[\psi]<0$. By definition, $ind(1) \geq m$. Since $m$ is arbitrary, we conclude.

\end{proof}
We finish the proof of Theorem \ref{th_morse_index} by showing that the Morse index of the solutions of \eqref{eq:cylinder} that satisfy (OC) is also infinite. This computation is based on the following proposition:

\begin{proposition} \label{negative} Let $v$ be a solution to \eqref{eq:cylinder} which satisfies the Oscillation Condition {(OC)} given in Definition \ref{def:osc} for certain $M>0$, $\e>0$. Then for any interval $I=[a,b]$ of length $|b-a| \geq 5M$, there exists a function $\phi \in C_0^{\infty}(\R)$ with $supp \ \phi \subset I$ such that
	\[
	\| \phi \|_{L^{\infty}} \leq \frac{1}{\delta} \mbox{ and }	Q_v[\phi]  \leq - \delta
	\]
	where $\delta>0$ does not depend on the choice of the interval $I$.
\end{proposition}

\begin{proof}

Take $I=[a,b]$ an interval as in the statement of the proposition. The proof is developed in several steps.

\medskip {\bf Step 1:} There exists $a< x_0 < x_1 < b$ such that $v'(x_j)=0$, $j=0,1$, and  
\begin{equation}\label{eq:integral_bounds}
\int_{x_0}^{x_1} (v'(t))^+ \, dt > 2 \e, \ \int_{x_0}^{x_1} (v'(t))^- \, dt > 2 \e.
\end{equation}
Here $f^+(x):=max\{f(x),0\}$ and $f^-(x):=max\{-f(x),0\}$ denote the functions positive and negative part of $f$.
\medskip

Take $t_i$, $i=1 \dots 6$, $t_1=a$, $t_6=b$, $t_{i+1} - t_i = \frac{b-a}{5}$, and define $I_i=[t_i, t_{i+1}]$, $i = 1 \dots 5$. By the Oscillation Condition (OC), there exists $y_i \in (t_i, t_{i+1})$ with $v(y_i)=1$, $i = 1 \dots 5$.

Observe now that the interval $[y_1, y_3]$ contains $I_2$. Again by (OC), there exists an interior global maximum and minimum of $v$ in $(y_1, y_3)$ which are greater than $1+\e$ and smaller than $1-\e$, respectively. Take $x_0 = \min \{t >y_1: \ v'(t)=0\}$.

We now reason analogously in the interval $[y_3, y_5]$, which contains $I_4$, and define $x_1= \max \{t <y_5: \ v'(t)=0\}$.

Let us take the minimum value of $v$ in $(y_1, y_3)$, achieved at $y$, and the maximum of $v$ in $(y_3, y_5)$, attained at $\bar{y}$. Then,

$$ 2 \e \leq v(\bar{y}) - v(y) = \int_{y}^{\bar{y}} v'(s) \, ds \leq \int_{y}^{\bar{y}} (v'(s))^+ \, ds \leq \int_{x_0}^{x_1} (v'(s))^+ \, ds.$$
By taking a maximum in $(y_1, y_3)$ and a minimum in $(y_3, y_5)$, we can reason as previously to obtain 

$$ 2 \e \leq \int_{x_0}^{x_1} (v'(s))^- \, ds,$$
as claimed.

\medskip {\bf Step 2:} Define $\eta: \R \to \R$ as:
$$ \eta(t)= \left \{ \begin{array}{ll} |v'(t)| & t \in [x_0, x_1], \\ 0 & \mbox{ otherwise. }  \end{array} \right. $$
Then $\eta \in H^s(\R)$ and $Q_v[\eta]< - \delta$, where $\delta >0$ depends only on $\e $ and $M$.

The assertion  $\eta \in H^s(R)$ is inmediate from its definition. Let us now compute the quadratic form \eqref{radialquad} on $\eta$.
We first claim that 
\begin{equation}\label{eq:quad_easier}
	Q_v[\eta]=\int_{x_0}^{x_1}\int_{x_0}^{x_1}K(t-\tau)\left(v'(t)v'(\tau)-|v'(t)v'(\tau)|\right)\,dt\,d\tau.
\end{equation}
Indeed, letting $J=[x_0,x_1]$, \eqref{radialquad} applied to $\eta$ reads

\begin{align*}
Q_v[\eta]&= \frac 1 2 \int_\R \int_\R \left(\eta(t)-\eta(\tau)\right)^2K(t-\tau)\, d\tau dt+\int_\R\left(1-\frac{n+2s}{n-2s}v(t)^\frac{4s}{n-2s}\right)\eta(t)^2dt \\
&=\frac{1}{2}\int_{J}\int_{J}\left(|v'(t)|-|v'(\tau)|\right)^2K(t-\tau)\,d\tau\,dt+\int_J\int_{\r\setminus J}v'(t)^2K(t-\tau)\,d\tau\,dt\\
&+\int_{J} \left(1-\frac{n+2s}{n-2s}v^{\frac{4s}{n-2s}}(t)\right){v'(t)}^2\,dt.
\end{align*}
Additionally, we can differentiate the equation \eqref{eq:cylinder} to get \[ Pv'(t)+\left(1-\frac{n+2s}{n-2s}v(t)^{\frac{4s}{n-2s}}\right)v'(t)=0. \]
Testing this new equation with 
\[ \varphi(t)= \left \{ \begin{array}{ll} v'(t) & t \in [x_0, x_1], \\ 0 & \mbox{ otherwise, }  \end{array} \right. \]
we obtain
\begin{align*}
\int_J\left(1-\frac{n+2s}{n-2s}v(t)^\frac{4s}{n-2s}\right)v'(t)^2dt&=-\int_J v'(t)\left(\text{P.V.} \int_\R (v'(t)-v'(\tau))K(t-\tau)d\tau\right)dt \\ &=-\int_Jv'(t)\lim_{\e\to 0}\int_{\R\setminus(t-\e,t+\e)}\hspace{-0.5cm}(v'(t)-v'(\tau))K(t-\tau)d\tau dt.
\end{align*}
For $R>0$ large enough, reasoning as in the proof of Lemma \ref{lemmaquad}, we get the existence of a constant $C=C(R)>0$ such that
\[
\abs{v'(t)\int_{\R\setminus(t-\e,t+\e)}(v'(t)-v'(\tau))K(t-\tau)d\tau} \leq C\abs{v'(t)},\hsp\text{for all}\hsp t\in J.
\]
Therefore, by the Dominated Convergence Theorem:
\begin{align*}
&\int_Jv'(t)\lim_{\e\to 0}\int_{\R\setminus(t-\e,t+\e)}\hspace{-0.5cm}(v'(t)-v'(\tau))K(t-\tau)d\tau dt = \lim_{\e\to0} \int_J\int_{\R\setminus(t-\e,t+\e)}\hspace{-0.5cm}v'(t)(v'(t)-v'(\tau))K(t-\tau)d \tau \\ &=\lim_{\e\to0} \left(\int_J\int_{J\setminus(t-\e,t+\e)}\hspace{-0.5cm}v'(t)(v'(t)-v'(\tau))K(t-\tau)d \tau+\int_J\int_{(\R\setminus J)\setminus(t-\e,t+\e)}\hspace{-0.5cm}v'(t)(v'(t)-v'(\tau))K(t-\tau)d \tau\right) \\
&= -\frac{1}{2}\int_{J}\int_{J}\left(v'(t)-v'(\tau)\right)^2K(t-\tau)\,d\tau\,dt
-\int_J\int_{\r\setminus J }\hspace{-0.15cm}v'(t)^2K(t-\tau)\,d\tau\,dt,
\end{align*}
where in the last step we have used Fubini's Theorem together with the integrability of $v'(t)(v'(t)-v'(\tau))K(t-\tau)$ in $\{(t,\tau)\in J^2:\,\abs{t-\tau}>\e\}.$ 
Finally, substituting:
\begin{align*}
Q_v[\eta]=&\frac{1}{2}\int_{J}\int_{J}\left(\left(|v'(t)|-|v'(\tau)|\right)^2-\left(v'(t)-v'(\tau)\right)^2\right)K(t-\tau)d\tau dt\\
=&\int_J\int_JK(t-\tau)\left(v'(t)v'(\tau)-|v'(t)||v'(\tau)|\right)d\tau dt.
\end{align*}
Let us define
\begin{equation}
	J^+=\{x\in J:\ v'(t)>0\} \text{ and } J^-=\{x\in J:\ v'(t)<0\},
\end{equation}
which are nonempty because $v$ satisfies \eqref{eq:integral_bounds}. Then, using \eqref{eq:quad_easier} we can bound the quadratic form as follows
\[
	\begin{split}
		Q_v[\eta]&=2\int_{J^+}\int_{J^-}K(t-\tau)\left(v'(t)v'(\tau)-|v'(t)v'(\tau)|\right)\,dt\,d\tau\\
		&=4\int_{J^+}\int_{J^-}K(t-\tau)v'(t)v'(\tau)\,dt\,d\tau< 4 K(10M) \int_{J^+}v'(t)\,dt\int_{J^-}v'(\tau)\,d\tau< -4 K(10M)\e^2.
	\end{split}
\]
The last two inequalities follow from the monotonicity of the kernel, the bound \linebreak $\abs{t-\tau}\leq 2\abs{x_1-x_0}<10M$ and \eqref{eq:integral_bounds}.

\medskip {\bf Step 3:} Conclusion.

\medskip We conclude since $C_0^{\infty}(I)$ is dense in $\{u  \in H^s(\R), \ supp \ u \in I\}$. 
\end{proof}

\begin{proof}[\textbf{Conclusion of the Proof of Theorem \ref{th_morse_index}}] By Proposition \ref{prop2}, it suffices to show that $ind(v) = +\infty$ for any solution $v$ of \eqref{eq:cylinder} satisfying the Oscillation Condition. In this case, we will work with a family of translations of the functions given by Proposition \ref{negative}.

\medskip 

Given $m \in \N$, take $m$ disjoint intervals $I_j$, $j=1, \dots, m,$ of length $5M$, and $\phi_j$ as in Proposition \ref{negative}. In the spirit of the proof of Proposition \ref{prop2}, we consider minimal distance between all the intervals, $d = \min \{|t_i-t_j|; t_i \in I_i, \ t_j \in I_j,\ i \neq j\}$, that will be taken large enough. We remark that the functions $\phi_j$ with $j=1,\ldots,m$ are linearly independent, since they have disjoint supports. 

\medskip 

Now, take any nontrivial linear combination of $\phi_j$, $\phi= \sum_{j=1}^m \lambda_j \phi_j$, and evaluate $Q_v[\phi]$:
\[
Q_v[\phi] = \sum_{j=1}^m \lambda_j^2 Q_v[\phi_k] + \sum_{i \neq j}^m \lambda_i \lambda_j A_v[\phi_i, \phi_j],
\]
where $A_v$ is the bilinear form associated to $Q_v$, that is,
\[
	A_v[\phi, \varphi]=	\frac{1}{2}\int_{\r} \int_{\r}\left(\phi(t)-\phi(\tau)\right)\left(\varphi(t)-\varphi(\tau)\right)K(t-\tau)\,dt\,d\tau+\int_{\r} \left(1-\frac{n+2s}{n-2s}v^{\frac{4s}{n-2s}}(t)\right)\phi(t)\varphi(t)\,dt.
	\] 
We estimate as follows:
\[ Q_v[\phi] \leq - \delta \sum_{j=1}^m \lambda_j^2  + \sum_{i \neq j}^m \lambda_i \lambda_j K(d) \left( \int_{\r} \int_{\r} |\phi_i(t)|\, |\phi_j(\tau)|\, dt \, d \tau \right)   \]\[ \leq - \delta \sum_{j=1}^m \lambda_j^2  + \sum_{i \neq j}^m \lambda_i \lambda_j K(d) \frac{(5M)^2}{\delta^2},\]
where $\delta>0$ is given by Proposition \ref{negative}. By taking $d$ sufficiently large, we can have that $Q_v[\phi]<0$, so $ind(v) \geq m$ by definition. Since $m \in \N$ was arbitrary, we conclude.
\end{proof}


\begin{thebibliography}{10}
	
	
	\bibitem{ACDFGW} W. Ao, H. Chan, A. DelaTorre, M. A. Fontelos, M. d. M. Gonz\'alez, J. Wei. On higher-dimensional singularities for the fractional Yamabe problem: A nonlocal Mazzeo-Pacard program. Duke Math. Journal Vol. 168, Number 17 (2019), 3297-3411. 
	\bibitem{survey} W. Ao, H. Chan, A. DelaTorre, M. A. Fontelos, M. d. M. Gonz\'alez, J. Wei. ODE-methods in nonlocal equations. Journal of Mathematical Study, Vol. 53, No. 4, pp. 370-401,(2020). 
	

	\bibitem{CafGidSpr} L.  Caffarelli, B. Gidas, and J. Spruck. Asymptotic symmetry and local behavior of semilinear elliptic equations with critical Sobolev growth. Comm. Pure Appl. Math., 42(3):271–297, 1989.
	
	\bibitem{CJSX} L. Caffarelli, Ti. Jin, Y.Sire, J. Xiong Local Analysis of Solutions of Fractional Semi-Linear Elliptic Equations with Isolated Singularities. Arch. Rational Mech. Anal. 213 (2014) 245–268. 
	
		\bibitem{ChaDlT2}  H. Chan and A. DelaTorre: From fractional Lane-Emden-Serrin equation -- existence, multiplicity and local behaviors via classical ODE -- to fractional Yamabe metrics with singularity of "maximal" dimension. Preprint. Available at arXiv:2109.05647
	
	
	\bibitem{ChaGon} S.-Y. A. Chang and M.d. M. Gonz\'alez. Fractional Laplacian in conformal geometry. Adv. Math., 226(2):1410–1432, 2011. 
	Linear Elliptic Equations with Isolated Singularities. Arch. Rational Mech. Anal. 213 (2014) 245–268. 

	\bibitem{CheLiOu}   W. Chen, C. Li, B. Ou. Classification of solutions for an integral equation. Comm. Pure Appl. Math. 59 (2006), no. 3, 330–343. 
	
	\bibitem{DG} A. DelaTorre and M.d.M Gonz\'alez. Isolated singularities for a semilinear equation for the fractional Laplacian arising in conformal geometry. Revista Matem\'atica Iberoamericana 34 (2018), no. 4. 
	\bibitem{DDGW}A. DelaTorre, M. del Pino, M.d.M Gonz\'alez and J. Wei. Delaunay-type singular solutions for the fractional Yamabe problem. Mathematische Annalen (2017),1-2.
	
		\bibitem{GonMazSir}  M. d. M. Gonz\'alez, R. Mazzeo, Y. Sire. Singular solutions of fractional order conformal Laplacians. J. Geom. Anal. 22 (2012), no. 3, 845–863. 
\bibitem{GonQin}	M. d. M. Gonz\'alez, J. Qing. Fractional conformal Laplacians and fractional Yamabe problems. Analysis and Partial	Differential Equations. 6 (2010).
	\bibitem{GonWan} M. d. M. Gonz\'alez, M. Wang. Further Results on the Fractional Yamabe Problem: The Umbilic Case. The Journal of Geometric Analysis 28 (2015): 22–60.
	
	
	\bibitem{GraZwo}C. R. Graham, M. Zworski. Scattering matrix in conformal geometry. Invent. Math. 152 (2003), no. 1, 89–118. 
	\bibitem{hitchhiker} {E. Di Nezza, G. Palatucci, E. Valdinoci}. {Hitchhiker’s guide to the fractional Sobolev spaces}. Bull. Sci. Math. 136(5), 521–573 (2012)
	

	\bibitem{KimMusWei} S. Kim, M. Musso, J. Wei. Existence theorems of the fractional Yamabe problem. Analysis $\&$ PDEs, 11(1), (2018) 75–113.
%
%

\bibitem{surveyRO}  Xavier Ros-Oton
Nonlocal elliptic equations in bounded domains: A survey.
Publicacions Matem\`{a}tiques,  Vol. 60, No. 1 (2016).
\bibitem{rudinbook}{W. Rudin}. Principles of Mathematical Analysis. McGraw-Hill (1953).

	\bibitem{Sch2} R. M. Schoen. Variational theory for the total scalar curvature functional for Rie- mannian metrics and related topics. In Topics in calculus of variations (Monteca- tini Terme, 1987), volume 1365 of Lecture Notes in Math., pages 120–154. Springer, Berlin, 1989. 
	\bibitem{Sch3} R. M. Schoen. On the number of constant scalar curvature metrics in a conformal class. In Differential geometry, volume 52 of Pitman Monogr. Surveys Pure Appl. Math., pages 311–320. Longman Sci. Tech., Harlow, 19
	\bibitem{SchYau}R. Schoen, S.T. Yau. Conformally flat manifolds, Kleinian groups and scalar curvature. Invent. Math. 92 (1988) 47–72. 
	
%

 \end{thebibliography}
\end{document}